\documentclass{article}
\usepackage{graphicx}
\usepackage{amsfonts}
\usepackage{amsmath}
\usepackage{amssymb}
\usepackage{indentfirst}
\usepackage{enumerate}
 \usepackage[colorlinks=true,citecolor=blue]{hyperref}
\usepackage{amsthm}
\usepackage{color}
\usepackage{natbib}
\setcitestyle{numbers,square}

\usepackage{comment}
\usepackage{authblk}

\newcommand{\vf}{\varphi}
\newcommand{\ve}{\varepsilon}
\theoremstyle{plain}
\newtheorem{theorem}{Theorem}
\newtheorem{corollary}[theorem]{Corollary}
\newtheorem{lemma}[theorem]{Lemma}
\newtheorem{proposition}[theorem]{Proposition}
\theoremstyle{definition}
\newtheorem{definition}{Definition}[section]
\newtheorem{example}{Example}[section]

\theoremstyle{remark}
\newtheorem{remark}{Remark}[section]

\numberwithin{equation}{section} \numberwithin{theorem}{section}

\usepackage{tikz-qtree}

\title{A Noether theorem for random locations}   

\begin{document}

\author[1]{Shunlong Luo}
\author[2]{Jie Shen}
\author[2]{Yi Shen}
\affil[1]{\footnotesize Academy of Mathematics and Systems Science, Chinese Academy of Sciences, China}
\affil[2]{\footnotesize Department of Statistics and Actuarial Science, University of Waterloo, Canada \normalsize}
\date{}
\maketitle

\begin{abstract}
We propose a unified framework for random locations exhibiting some probabilistic symmetries such as stationarity, self-similarity, etc. A theorem of Noether's type is proved, which gives rise to a conservation law describing the change of the density function of a random location as the interval of interest changes. We also discuss the boundary and near boundary behaviour of the distributions of the random locations.
\end{abstract}


\section{Introduction}
The famous Noether theorem in mathematical physics \cite{noether1918} shows that each differentiable symmetry of a system corresponds to a conservation law. The most important and immediate examples include translation in space and the conservation of momentum, translation in time and the conservation of energy, rotation in space and the conservation of angular momentum, \textit{etc}. A thorough review of Noether theorem can be found in the book by Kosmann-Schwarzbach \cite{kosmann2011noether}.

Since the last two decades of the twentieth century, various works have been carried out to extend Noether theorem to stochastic settings. Just to name a few, Yasue \cite{yasue1981stochastic} proposed a theory for stochastic calculus of variations, and obtained a corresponding generalization of the Noether theorem. Misawa \cite{misawa1994conserved} considered the conservative quantities and symmetry for stochastic dynamical systems described by certain type of stochastic differential equations. Thieullen and Zambrini proved a version of the Noether theorem, in which they associated a function giving a martingale to each family of transformations exhibiting certain symmetry \cite{thieullen1997symmetries}. They also extended the Noether theorem to diffusion processes in $\mathbb R^3$ whose diffusion matrix is proportional to identity \cite{thieullen1997probability}. Entering the new century, van Casteren \cite{van2003hamilton} obtained a version of the stochastic Noether theorem using the ideas and backgrounds from stochastic control. More recently, Baez and Fong \cite{Baez2013} considered Markov processes and found an analogy of the classical Noether theorem in this setting. Along this direction, Gough, Ratiu and Smolyanov \cite{gough2015noether} gave a Noether theorem for dissipative quantum dynamical semi-groups. Another scenario where an external random force exists was studied by Luzcano and de Oca \cite{lezcano2017stochastic}.

The random locations of stochastic processes exhibiting certain probabilistic symmetries have been studied in a series of works in the past years. In \cite{samorodnitsky2013intrinsic}, Samorodnitsky and Shen introduced a large family of random locations called ``intrinsic location functionals'', which include the location of the path supremum, the first/last hitting time to a fixed level, \textit{etc}. It was shown that the distribution of any random location in this family for a stationary process must satisfy a specific set of conditions. Similar results were later established between a subclass of intrinsic location functionals and stochastic processes with stationary increments \cite{shen2016random}. In \cite{shen2018self}, the stochastic processes combining both a scaling symmetry and a stationarity of the increments were studied, and it is shown that stronger conditions hold for the distribution of its path supremum over an interval.

As the research of random locations progressed, it became clearer and clearer that there is a general correspondence between probabilistic symmetries and classes of random locations, such that the distributions of the random locations behave in a very specific way under the corresponding symmetry. Indeed, it is not difficult to see that the setting for the random locations of stochastic processes having probabilistic symmetries is similar to the settings in which Noether theorems hold, in that they are both systems with infinitesimally generated symmetries. This observation leads to the question as whether a result of Noether's type exists for the random locations. There is, however, a critical difference: in the case of random locations, the symmetries are only in the distributional sense. While the overall distribution of the processes, hence also the distributions of the random locations, remain invariant after the corresponding transformations, the values of the locations do evolve with the transformations in each realization. As a result, the mathematical tools used to derive the Noether theorems for deterministic systems can not be applied to get similar results here. It turns out that the methods developed in the literatures previously mentioned are not helpful as well.

The goal of this paper is, therefore, to provide a framework which contains the aforementioned random locations and probabilistic symmetries as special cases, and in which a Noether theorem can be established. To this end, we generalize the notion of random location by dissociating it from the paths of stochastic processes. More precisely, the random locations are no longer functionals of the paths as in \cite{samorodnitsky2013intrinsic, shen2016random, shen2018self}, but special elements in a point process which may or may not be related to a stochastic process in continuous time. Another point process is then constructed, and we show that the distribution of the random locations can be expressed in terms of the control measure of the latter point process. Finally, a conservation law appears using a function derived from the control measure.

The rest of this paper is organized as follows. In Section 2 we introduce the basic settings and definitions, with examples making connections to the existing literatures. In Section 3 we state and prove the main results, including the Noether theorem as a conservation law when the interval of interest moves along a flow, and its consequences, such as a constraint on the total variation of the density function of the random locations. Section 4 completes the paper by analyzing the behavior of the random locations at or near the boundaries of the interval of interest.

\section{Basic settings}
Here and throughout the paper, let $\mathcal I$ be the collection of all the non-degenerate compact intervals on $\mathbb R$. Let $\bar{\mathbb R}=\mathbb R\cup\{\infty\}$, and equip it with the $\sigma-$field $\bar{\mathcal B}=\sigma(\mathcal B(\mathbb R),\{\infty\})$. That is, we treat $\infty$ as a separate point and take the Borel $\sigma-$field of the extended topology.

\begin{definition}\label{def:irl}
A stochastic process $\{L(I)\}_{I\in\mathcal I}$ indexed by compact intervals and taking values in $\bar{R}$ is called an intrinsic random location, if it satisfies the following conditions:
\begin{enumerate}
\item For every $I\in \mathcal I$, $L(I)\in I\cup\{\infty\}$.
\item (Stability under restriction) For every $I_1, I_2\in\mathcal I$, $I_2\subseteq I_1$, if $L(I_1)\in I_2$, then $L(I_1)=L(I_2)$.
\item (Consistency of existence) For every $I_1, I_2\in\mathcal I$, $I_2\subseteq I_1$, if $L(I_2)\neq \infty$, then $L(I_1)\neq\infty$.
\end{enumerate}
\end{definition}

Intuitively, the value $\infty$ is used to deal with the case where a random location is not well-defined on a given interval for certain realization. For example, if the random location is defined as the first hitting time of a continuous-time stochastic process to certain level, then it is possible that the process does not hit the level in the given interval. In this case we will assign $\infty$ as the value of the random location.

Let $\vf=\{\vf^t\}_{t\in\mathbb R}$ be a flow on $\mathbb R$. That is, $\{\vf^t\}_{t\in\mathbb R}$ is a family of real-valued functions defined on $\mathbb R$, satisfying $\vf^0=Id$ and $\vf^s\circ \vf^t=\vf^{s+t}$ for $s,t\in\mathbb R$. We further assume that
\begin{equation}\label{cond.C1}\vf^t(x)=\vf(x,t)\in C^{1,1}(\mathbb R\times \mathbb R);\end{equation}
\begin{equation}\label{cond.isolated}\text{the fixed points }\Phi_0:=\{x: \vf^t(x)\equiv x\}\text{ are isolated.}\end{equation}

In many cases, it will be convenient to consider the extended real line $\mathbb R\cup\{-\infty, \infty\}$ and the set of extended fixed points $\bar{\Phi}_0=\Phi_0\cup\{-\infty, \infty\}$. Two points $\alpha, \beta, \alpha<\beta$ are called consecutive in $\bar{\Phi}_0$, if $\alpha, \beta\in \bar{\Phi}_0$, and $(\alpha,\beta)\cap\bar{\Phi}_0=\phi$. Note that since there is no fixed point between $\alpha$ and $\beta$, and $\vf$ is continuous, $\vf^t(x)$ must be monotone in $t$ for any fixed $x\in (\alpha, \beta)$ and increasing in $x$ for any fixed $t\in\mathbb R$. In particular, for every fixed $x\in(\alpha,\beta)$, $\vf^{\cdot}(x)$ is a bijection from $\mathbb R$ to $(\alpha,\beta)$.

An intrinsic random location is called $\vf$-stationary, if its distribution is compatible with the flow $\vf$, more precisely, if $\vf^t(L([a,b]))\stackrel{d}{=}L([\vf^t(a), \vf^t(b)])$ for every $t\in\mathbb R$ and $a,b\in\mathbb R, a<b$. It is called stationary if the flow is the translation $\vf^t(x)=x+t$.

\begin{remark}\label{rem:transform}
Due to the continuity of $\vf$, a $\vf$-stationary intrinsic random location, restricted to the open interval between two consecutive extended fixed points of $\vf$, can be easily transformed into a stationary intrinsic random location using a transformation. More precisely, let $L$ be a $\vf$-stationary intrinsic random location and $\alpha, \beta$ be two consecutive points in $\bar{\Phi}_0$. Fix any $x_0\in(\alpha, \beta)$. Then $\vf^{t}(x_0)$ is a continuous monotone function in $t$ with $\lim_{t\to -\infty}\vf^t(x_0)=\alpha$ and $\lim_{t\to \infty}\vf^t(x_0)=\beta$, or symmetrically, $\lim_{t\to -\infty}\vf^t(x_0)=\beta$ and $\lim_{t\to \infty}\vf^t(x_0)=\alpha$. As a result, we can define a transform $\tau: (\alpha, \beta)\to\mathbb R$ by
$$
\vf^{\tau(x)}(x_0)=x.
$$
That is, $\tau(x)$ is the time it takes to go from $x_0$ to $x$ following the flow $\vf$, or from $x$ to $x_0$ if its value is negative. Hence for any $x, y\in (\alpha, \beta)$, $\vf^{\tau(y)-\tau(x)}(x)=y$. Differentiating at $y=x$, we have
\begin{equation}\label{e:deriv}
\dot{\vf}^0(x)=\left(\frac{d\tau}{dx}\right)^{-1}(x),
\end{equation}
where $\dot{\vf}^t(x)=\frac{\partial \vf(x,t)}{\partial t}\big|_{x,t}$. Moreover, we have identity
\begin{equation}\label{e:identity}
\tau(x)=\tau((\vf^t)^{-1}(x))+t
\end{equation}
for $x\in(\alpha, \beta)$ and $t\in \mathbb R$.

Since $\tau$ is a bijection, its inverse $\tau^{-1}$ is well-defined. Define $L'$ by
$$
L'(I)=\tau(L(\tau^{-1}(I))), \quad I\in\mathcal I.
$$
It is elementary to check that if $L$ is $\vf$-stationary, then such defined $L'$ is a stationary intrinsic random location. Consequently, all the claims regarding a $\vf$-stationary intrinsic random location can be transformed into corresponding claims regarding stationary intrinsic random locations, and we only need to prove the latter ones.
\end{remark}

As explained in Introduction, the definition of intrinsic random location is motivated by the random locations of stochastic processes studied in previous literatures \cite{samorodnitsky2013intrinsic, shen2016random, shen2018self}. Therefore, it is not surprising that one important way to obtain $\vf$-stationary intrinsic random locations is through the stochastic processes exhibiting some probabilistic symmetry under $\vf$, and to define the random location as a functional which is determined by the path of the process and compatible with $\vf$. For example, let the flow be the translation $\vf^t(x)=x+t$. Correspondingly, we have the (strictly) stationary processes as the family of processes whose distributions are invariant under $\vf$. In this case, let $H$ be a space of functions closed under translation, equipped with the cylindrical $\sigma$-field, and consider a mapping $L_H: \mathcal I\times H \to \bar{\mathbb R}$ satisfying
\begin{enumerate}
\item $L_H(I, \cdot): H\to \bar{\mathbb R}$ is measurable;
\item $L_H(I, f)\in I\cup\{\infty\}$ for every $f\in H$;
\item For every $I_1, I_2\in\mathcal I$, $I_2\subseteq I_1$ and every $f\in H$, if $L_H(I_1, f)\in I_2$, then $L_H(I_2, f)=L_H(I_1, f)$;
\item For every $I_1, I_2\in\mathcal I$, $I_2\subseteq I_1$ and every $f\in H$, if $L_H(I_2, f)\neq \infty$, then $L_H(I_1, f)\neq \infty$;
\item $L_H(I, f)=L_H(I-t, f\circ\vf^t)+t$ for any $f\in H$, \text{ where } $I-t:=\{x\in\mathbb R: x+t\in I\}$.
\end{enumerate}
Conditions 2,3 and 4 correspond to the three conditions in the definition for intrinsic random locations, while Condition 5 requires the random location to be compatible with translation. Then it is easy to check that the random location $L$ defined by
$$
L(I)(\omega)=L_H(I,X(\cdot, \omega))
$$
is a stationary intrinsic random location if $\mathbf X=\{X(t,\omega)\}_{t\in\mathbb R}$ is a stationary process with sample paths in $H$. Such a mapping like $L_H$ was introduced in \cite{samorodnitsky2013intrinsic}, where its relation to stationarity has also been studied in detail. We note that this is indeed a very large family of random locations, including the location of the path supremum/infimum over an interval, the first/last hitting time to certain level, among many others.

Other probabilistic symmetries of stochastic processes which can be used to define intrinsic random locations stationary with respect to certain flow  include self-similarity, isometry (in higher dimensional domains), stationarity of the increments, \textit{etc}. They have been discussed respectively in the sequence of papers \cite{shen2018self, shen2013thesis, shen2016random}. Two cases are special and worth some more mention.

First, even for a same $\vf$, there can be various ways to construct $\vf$-stationary intrinsic random locations from stochastic processes. For instance, still consider the translation. If instead of the distribution of the process, we only require the distribution of the increments of the process to be translation invariant, then the resulting family of processes is the family of processes with stationary increments, which is strictly larger compared to the family of stationary processes. As a price for the relaxation of the condition on the side of processes, a stronger assumption needs to be imposed to the mapping $L_H$. More precisely, $L_H$ now needs to be invariant under vertical shift of the path: $L_H(I,f)=L_H(I,f+c)$ for any $f\in H$ and $c\in \mathbb R$. It has been shown in \cite{shen2016random} that similar results as in \cite{samorodnitsky2013intrinsic} hold between such random locations and stochastic processes with stationary increments.

Second, different symmetries can be combined together. For instance, due to the Lamperti transformation (see, for example, \cite{embrechts2002selfsimilar}), self-similarity by itself does not give any result which is new in nature. However, as shown in \cite{shen2018self}, when it is combined with the stationarity of the increments, stronger distributional properties can be derived for the random locations which are compatible with both scaling and translation.

It should be pointed out that although many $\vf$-intrinsic random locations are defined using certain continuous-time stochastic processes, such processes are not an indispensable part of the construction. It is in this sense that the current framework is a generalization of those used in previous works, where the definition of the random location does require a continuous-time process.

\begin{example} \label{ex:point}
Let $\{X_i\}_{i\in \mathbb Z}$ be a strictly increasing sequence of random variables such that the point process on $\mathbb R$ determined by it, $\sum_i\delta_{X_i}$, where $\delta_x(A)=\mathbf 1_{\{x\in A\}}$, is a stationary point process. Let $\{Y_i\}_{i\in \mathbb Z}$ be a discrete-time stationary process. Then one can define random locations such as
$$
L_1(I)=\sup\{X_i: X_i\in I\}
$$
and
$$
L_2(I)=\inf \{X_i: X_i\in I, Y_i=\sup_{j: X_j\in I} Y_j\},
$$
where the tradition $\inf(\phi)=\sup(\phi)=\infty$ is used. Intuitively, among all the points with the first coordinate in $I$, $L_1$ takes the largest first coordinate, while $L_2$ takes the first coordinate of the point with the largest second coordinate. The infimum in the definition of $L_2$ is to deal with the case where the supremum is achieved in multiple points. If in addition, we have $P(Y_i=Y_j)=0$ for all $i,j$, then the infimum can be removed. It is easy to check that both $L_1$ and $L_2$ are stationary intrinsic random locations.
\end{example}

The point process in example \ref{ex:point} can be regarded as a one-dimensional point process given by $\{X_i\}_{i\in \mathbb Z}$ in which each point $X_i$ also gets a label $Y_i$ in a stationary way. The following example is more ``higher dimensional'' and geometrical in nature.

\begin{example}
Consider a stationary random tessellation of $\mathbb R^2$ such as the Gilbert tessellation. For any compact intervals $I$ and $I'$, among all the pieces of the tessellation for which the geometric center is located in $I\times I'$, take the one with the largest area. Then the first or the second coordinate of its geometric center is a stationary intrinsic random location indexed by $I$ or $I'$, respectively, where we again follow the tradition to assign value $\infty$ when no piece has its center in $I\times I'$.
\end{example}

\section{Main results}
We start this section by introducing some preparatory results.

The stability under restriction property in Definition \ref{def:irl} implies the following simple yet useful comparison lemma.

\begin{lemma}\label{lem:comparison}
Let $L$ be an intrinsic random location. Then for any $I_1, I_2\in \mathcal I$ such that $I_2\subseteq I_1$ and any $I\subseteq I_2$, $P(L(I_1)\in I)\leq P(L(I_2)\in I)$.
\end{lemma}
\begin{proof}
By stability under restriction, $L(I_1)\in I\subseteq I_2$ implies $L(I_2)=L(I_1)\in I$, hence the result.
\end{proof}

 The distribution of a stationary intrinsic random location $L=L(I)$ is absolutely continuous in the interior of the interval $I$. Indeed, the next proposition does not only show the absolute continuity, but also provides an upper bound for the density. It was first proved in \cite{samorodnitsky2013intrinsic} for the stationary processes and random locations which are compatible with translation. Here we include a short proof of a modified version for the sake of completeness.
\begin{proposition}\label{prop:bound}
Let $L$ be a stationary intrinsic random location. For any $a<x<b$ and $0<\epsilon<\min\left\{x-a,b-x\right\}$,
\begin{align}\label{abso}
P(L([a,b])\in (x,x+\epsilon])\leq 2\epsilon\max\{\frac{1}{x-a},\frac{1}{b-x}\}.
\end{align}
\end{proposition}

\begin{proof}
Suppose that, to the contrary, (\ref{abso}) fails for some $a, b, x$ and $\epsilon$. That is,
\begin{align*}
P(L([a,b])\in (x,x+\epsilon])>2\epsilon\max\{\frac{1}{x-a},\frac{1}{b-x}\}.
\end{align*}
Without loss of generality, assume $x-a\leq b-x$. Then
\begin{align*}
P\left(L([a,x])\in (x-y_i,x+\epsilon-y_i]\right)&=P\left(L([a+y_i,x+y_i])\in (x,x+\epsilon]\right)\\
&\geq P(L([a,b])\in (x,x+\epsilon])\\
& >2\epsilon\max\{\frac{1}{x-a},\frac{1}{b-x}\}\\
& =\frac{2\epsilon}{x-a}
\end{align*}
for $y_i=i\epsilon$, $i=1,\dots,\lfloor \frac{x-a}{\epsilon}\rfloor$.
Since $\frac{x-a}{\epsilon}\geq 1$, $\lfloor \frac{x-a}{\epsilon}\rfloor\geq \frac{x-a}{2\epsilon}$. Hence we have
\begin{align*}
1&\geq\sum^{\lfloor \frac{x-a}{\epsilon}\rfloor}_{i=1}P\left(L([a,x]) \in (x-y_i,x+\epsilon-y_i]\right)\\
&> \lfloor \frac{x-a}{\epsilon}\rfloor \frac{2\epsilon}{x-a}\geq 1.
\end{align*}
Contradiction. A similar contradiction can be derived for the case where $x-a> b-x$. Hence (\ref{abso}) is proved.
\end{proof}

As a consequence of Proposition \ref{prop:bound}, we also have the following continuity result.

\begin{lemma}\label{lem:continuity}
Let $L$ be a stationary intrinsic random location. Then for any $u,v\in\mathbb R, u<v$, $P(L([a,b])\in[u,v])$ is continuous in $a$ and $b$ for $a<u$ and $b>v$.
\end{lemma}

\begin{proof}
By symmetry, it suffices to prove that $P(L([a,b])\in[u,v])$ is continuous in $a$ for $a<u$. For $\varepsilon\in\left(0,\frac{u-a}{2}\right)$, we have
\begin{align*}
0\leq & P(L([a+\varepsilon, b])\in[u,v])-P(L([a-\varepsilon, b])\in[u,v])\\
= & (P(L([a-\varepsilon, b])\in[a-\varepsilon, u))-P(L([a+\varepsilon, b])\in[a+\varepsilon, u)))\\
& -(P(L([a+\varepsilon, b])\in(v,b])-P(L([a-\varepsilon, b])\in(v,b]))\\
& -(P(L([a+\varepsilon, b])=\infty)-P(L([a-\varepsilon, b])=\infty))\\
\leq & P(L([a-\varepsilon, b])\in[a-\varepsilon, u))-P(L([a+\varepsilon, b])\in[a+\varepsilon, u)),
\end{align*}
where the inequalities come from Lemma \ref{lem:comparison}. Also, by stationarity and Lemma \ref{lem:comparison},
\begin{multline*}
P(L([a+\varepsilon, b])\in[a+\varepsilon, u))=P(L([a-\varepsilon, b-2\varepsilon])\in[a-\varepsilon, u-2\varepsilon))\\
\geq P(L([a-\varepsilon, b])\in[a-\varepsilon, u-2\varepsilon)).
\end{multline*}
Hence
\begin{align}
& P(L([a+\varepsilon, b])\in[u,v])-P(L([a-\varepsilon, b])\in[u,v])\nonumber\\
\leq & P(L([a-\varepsilon, b])\in[a-\varepsilon, u))-P(L([a-\varepsilon, b])\in[a-\varepsilon, u-2\varepsilon))\nonumber\\
= & P(L([a-\varepsilon, b])\in[u-2\varepsilon, u)).\label{e:boundbyf}
\end{align}
By Proposition \ref{prop:bound}, $P(L([a-\varepsilon, b])\in[u-2\varepsilon, u))\leq P(L([a, b])\in[u-2\varepsilon, u)) \to 0$ as $\varepsilon\to 0$. Thus we conclude that $P(L([a,b])\in[u,v])$ is continuous in $a$ for $a<u$.
\end{proof}

In order to introduce a point process which will play an essential role in deriving the main results, we first show that each intrinsic random location gives a partial order among the potential values of the random location. Similar idea originated in \cite{shen2016random}. The proof is however different due to the difference in settings. More precisely, let $L$ be an intrinsic random location. Define the random set $S:=\{x\in\mathbb R: x=L(I) \text{ for some }I\in\mathcal I\}$. Define a binary relation ``$\preceq$'' on $S$:
$$
x\preceq y \quad \text{ if there exists }I\in\mathcal I, \text{ such that }x,y\in I, L(I)=y.
$$
Intuitively, $x\preceq y$ if both points are in a same interval, and the location falls on $y$ rather than on $x$.

\begin{lemma}\label{lem:partialorder}
$\preceq$ is a partial order.
\end{lemma}

\begin{proof}
It is easy to see that $\preceq$ is reflexive. It is antisymmetric since for any $I$ containing $x$ and $y$ and satisfying $L(I)=x$ or $L(I)=y$, $L(I)=L([x\wedge y,x\vee y])$ by the stability under restriction property in Definition \ref{def:irl}. As a result, $x\preceq y$ if and only if $L([x\wedge y, x\vee y])=y$. Finally, if $x\preceq y$ and $y\preceq z$, then by Definition \ref{def:irl},
$$
L([x\wedge y, x\vee y]\cup[y\wedge z, y\vee z])\in\{L([x\wedge y, x\vee y]), L([y\wedge z, y\vee z])\}=\{y,z\}\subset [y\wedge z, y\vee z]
$$
Again by the stability under restriction property, we must have $L([x\wedge y, x\vee y]\cup[y\wedge z, y\vee z])=L([y\wedge z, y\vee z])=z$, hence $x\preceq z$.
\end{proof}

For each $x\in S$, define $l_x:=\sup\{y\in S: y<x, x\preceq y\}$ and $r_x:=\inf\{y\in S: y>x, x\preceq y\}$. Intuitively, $l_x$ and $r_x$ are the farthest locations to the left and to the right of the point $x$ such that no point in $S$ between this location and $x$ has a higher order than $x$ according to $\preceq$. It is easy to see that if in addition, there exists $[a,b]\in\mathcal I$ such that $x=L([a,b])$ and $x\in (a,b)$, then $l_x\leq a<x$ and $r_x\geq b>x$. Thus, for every such $x$, the point in $\mathbb R^3$ defined by $\epsilon_x:=(l_x, x, r_x)$ falls in the area $E:=\{(z_1,z_2,z_3): z_1<z_2<z_3\}$. Let $\mathcal E$ be the collection of such points:
 $$
 \mathcal E=\{\epsilon_x=(l_x, x, r_x): x\in S, l_x<x<r_x\},
 $$
 then the (random) counting measure determined by $\mathcal E$, denoted by $\xi:=\sum_{x\in\mathcal E}\delta_{\epsilon_x}$, forms a point process in $E$. Since $l_x<a, r_x>b$ and $x\in(a,b)$ implies $L([a,b])=x$, $\mathcal E$ has at most one point in $(-\infty,a)\times (a,b)\times (b,\infty)$ for any $a,b\in\mathbb R, a<b$, hence the point process $\xi$ is $\sigma-$finite. Denote by $\eta$ its control measure, \textit{i.e.}, $\eta(A)=E(\xi(A))$ for any $A\in \mathcal B(E)$, where $\mathcal B(E)$ is the Borel $\sigma-$field on $E$.

\begin{theorem}\label{thm:relation}
Let $L$ be a stationary intrinsic random location, and $\eta$ be the control measure of the point process $\xi$ defined for $L$ as above. Then for any $a<u<v<b$,
\begin{equation}\label{eq:relation}
P(L([a,b])\in [u,v])=\eta((-\infty, a)\times (u,v)\times (b,\infty))=\eta((-\infty, a]\times [u,v]\times [b,\infty)).
\end{equation}
\end{theorem}

\begin{remark}
Theorem \ref{thm:relation} serves for three purposes. First, it builds a connection between the distribution of a stationary intrinsic random location and the control measure of the point process related to it. Second, it also shows that the planes in $E$ with one of the three coordinates fixed are always null sets under $\eta$. As a result, one does not need to pay special attention to the openess/closedness of the boundaries of the intervals for the coordinates. Finally, since $L$ is stationary, \textit{i.e.}, $P(L([a,b])\in[u,v])=P(L([a+c,b+c])\in[u+c,v+c])$ for all $a\leq u<v\leq b$ and $c\in \mathbb R$, and the sets of the form $(-\infty, a]\times [u,v]\times [b,\infty)$ generate $\mathcal B(E)$, the measure $\eta$ is invariant under translation along the direction $(1,1,1)$. We formulate this result as the following corollary, the proof of which is obvious and omitted.
\end{remark}

\begin{corollary}\label{cor:invariance}
Let $A\in\mathcal B(E)$. Then $\eta(A)=\eta(A+c)$ for any $c\in\mathbb R$, where $A+c=\{(z_1,z_2,z_3): (z_1-c, z_2-c, z_3-c)\in A\}$.
\end{corollary}

\begin{proof}[Proof of Theorem \ref{thm:relation}]
If $x=L([a,b])\in[u,v]$, then $x\in S$, $l_x\leq a$, and $r_x\geq b$. Note that it is possible that $l_x=a$ (resp. $r_x=b$), since $a$ (resp. $b$) can be the limit of an increasing (resp. decreasing) sequence of points in $S$ with higher orders than $x$ according to $\preceq$, while the endpoint itself is not in $S$ or does not have a higher order than $x$. Meanwhile, if there exists $x\in[u,v]\cap S$ such that $l_x<a$ and $r_x>b$, then we must have $x=L([a,b])$. Therefore,
\begin{align*}
& P(L([a-\ve, b+\ve])\in[u,v]) \leq \eta((-\infty, a)\times [u,v]\times (b, \infty))\leq P(L([a,b])\in[u,v])\\
\leq &  \eta((-\infty, a]\times [u,v]\times [b, \infty)) \leq P(L([a+\ve,b-\ve])\in[u,v]).
\end{align*}
The control measure $\eta$ appears in the above expression because there can be at most one point in $\mathcal E$ in the area $(-\infty, a]\times [u,v]\times [b, \infty)$. In this case the expectation coincides with the corresponding probability.

By Lemma \ref{lem:continuity},
$$
\lim_{\ve\downarrow 0}P(L([a-\ve, b+\ve])\in[u,v])=\lim_{\ve\downarrow 0}P(L([a+\ve, b-\ve])\in[u,v]),
$$
hence we must have
\begin{equation}\label{e:relation}
\eta((-\infty, a)\times[u,v]\times (b, \infty))=P(L([a,b])\in[u,v])= \eta((-\infty, a]\times [u,v]\times [b, \infty)).
\end{equation}

Finally, by Proposition \ref{prop:bound}, $L([a,b])$ is continuously distributed on $(a,b)$, hence $P(L([a,b])\in[u,v])$ is continuous in $u$ and $v$, so is $\eta((-\infty, a)\times [u,v]\times (b, \infty))$. Therefore, $\eta((-\infty, a)\times [u,v]\times (b, \infty))=\eta((-\infty, a)\times(u,v)\times (b, \infty))$.
\end{proof}

For a stationary intrinsic random location $L$, $a<u<v<b$ and any $\ve>0$, define
$$
M_{\ve, [u,v]}=P(L([a,b])\in [a,a+\ve), L([a+\ve, b+\ve])\in [u,v])
$$
and
$$
N_{\ve, [u,v]}=P(L([a+\ve,b+\ve])\in (b,b+\ve], L([a,b])\in [u,v]).
$$
Further define $\mu_{\ve, [u,v]}$ to be the conditional distribution of $L([a+\ve, b+\ve])$ given $L([a,b])\in [a,a+\ve)$ and $L([a+\ve, b+\ve])\in [u,v]$, and $\nu_{\ve,[u,v]}$ to be the conditional distribution of $L([a, b])$ given $L([a+\ve,b+\ve])\in (b,b+\ve]$ and $L([a, b])\in [u,v]$, if $M_{\ve, [u,v]}$ and $N_{\ve, [u,v]}$ are strictly positive. If $M_{\ve, [u,v]}=0$ or $N_{\ve,[u,v]}=0$, define the corresponding $\mu_{\ve,[u,v]}$ or $\nu_{\ve,[u,v]}$ to be the null measure.

Let $\mu^{(a,b)}$ and $\nu^{(a,b)}$ be measures on $(a,b)$ (equipped with the Borel $\sigma-$field) given by
\begin{equation}\label{e:mu}
\mu^{(a,b)}([w,y))=\eta \left(\{(z_1,z_2,z_3): z_1\in[a,a+1), z_2\in[z_1+w-a, z_1+y-a), z_3\in (z_1+b-a, \infty)\}\right)
\end{equation}
and
\begin{equation}\label{e:nu}
\nu^{(a,b)}([w,y))=\eta \left(\{(z_1,z_2,z_3): z_1\in(-\infty, z_3+a-b), z_2\in[z_3+w-b, z_3+y-b), z_3\in (b, b+1]\}\right)
\end{equation}
for all $w,y\in(a,b)$, $w<y$, where $\eta$ is the control measure of the point process $\xi$ corresponding to $L$ as defined previously. Denote by $\mu^{(a,b)}|_{[u,v]}$ and $\nu^{(a,b)}|_{[u,v]}$ the restriction of the measures $\mu^{(a,b)}$ and $\nu^{(a,b)}$ on $[u,v]$, respectively.

Our last preparation before proceeding to the proof of the main result is the following proposition.

\begin{proposition}\label{prop:limit}
Let $L$ be a stationary intrinsic random location. For $a<u<v<b$, let $M_{\ve, [u,v]}$, $N_{\ve, [u,v]}$, $\mu_{\ve, [u,v]}$ and $\nu_{\ve, [u,v]}$ be defined as above. Then $\frac{1}{\ve}M_{\ve, [u,v]}\mu_{\ve, [u,v]}$ and $\frac{1}{\ve}N_{\ve, [u,v]}\nu_{\ve, [u,v]}$ converge vaguely as $\ve\to 0$ to $\mu^{(a,b)}|_{[u,v]}$ and $\nu^{(a,b)}|_{[u,v]}$, respectively.
\end{proposition}

\begin{proof}
By symmetry it suffices to prove the convergence for $\frac{1}{\ve}M_{\ve, [u,v]}\mu_{\ve, [u,v]}$ as $\ve\to 0$. For any $\ve\in(0,b-a)$, define measure $\lambda_\ve$ on $[a+\ve, b)$ by
$$
\lambda_\ve(A)=P(L([a,b])\in [a,a+\ve), L([a+\ve, b+\ve])\in A), \quad A\in\mathcal B([a+\ve, b)),
$$
then it is easy to see that for any $\ve<u-a$ and $A'\in \mathcal B([u, v])$,
$$
M_{\ve, [u,v]}\mu_{\ve, [u,v]}(A')=P(L([a,b])\in [a,a+\ve), L([a+\ve, b+\ve])\in A')=\lambda_\ve(A').
$$
Hence it suffices to prove that $\frac{1}{\ve}\lambda_{\ve}([w, y))$ converges to $\mu^{(a,b)}([w,y))$ for any $w,y\in(a,b), w<y$.

Note that $L([a,b])\in [a,a+\ve)$ and $L([a+\ve, b+\ve])\in [w,y)$ implies that there exists a point $x\in [w,y)\cap S$, such that $l_x\in [a, a+\ve]$ and $r_x\in [b+\ve, \infty)$. Meanwhile, the existence of a $x\in [w,y)\cap S$ satisfying $l_x\in (a, a+\ve)$ and $r_x\in (b+\ve, \infty)$ would guarantee that $L([a,b])\in [a,a+\ve)$ and $L([a+\ve, b+\ve])\in [w,y)$. Therefore, we have
$$
\eta( (a,a+\ve)\times [w,y)\times (b+\ve, \infty))\leq  \lambda_\ve([w,y))\leq \eta([a,a+\ve]\times [w,y)\times [b+\ve, \infty)).
$$

By Theorem \ref{thm:relation}, the boundaries of the intervals are negligible under $\eta$. Hence
$$
\lambda_\ve([w,y))=\eta([a,a+\ve)\times [w,y)\times (b+\ve, \infty)).
$$
For $\ve=\frac{1}{n}$, $n\in\mathbb N$, by Corollary \ref{cor:invariance}, we have

\begin{align*}
& \frac{1}{\ve}\eta([a,a+\ve)\times [w,y)\times (b+\ve, \infty))\\
=& n\eta\left(\left[a,a+\frac{1}{n}\right)\times [w,y)\times \left(b+\frac{1}{n}, \infty\right)\right)\\
=& \sum_{i=0}^{n-1}\eta\left(\left[a+\frac{i}{n},a+\frac{i+1}{n}\right)\times\left[w+\frac{i}{n},y+\frac{i}{n}\right)\times \left(b+\frac{i+1}{n}, \infty\right)\right).
\end{align*}
Note that the set
$$
\bigcup_{i=0}^{n-1}\left(\left[a+\frac{i}{n},a+\frac{i+1}{n}\right)\times \left[w+\frac{i}{n},y+\frac{i}{n}\right)\times \left(b+\frac{i+1}{n}, \infty\right)\right)
$$
contains
$$
\left\{(z_1,z_2,z_3): z_1\in[a,a+1), z_2\in\left[z_1+w-a, z_1+y-a-\ve\right), z_3\in(z_1+b-a+\ve, \infty) \right\},
$$
and is contained in
$$
\left\{(z_1,z_2,z_3): z_1\in[a,a+1), z_2\in\left[z_1+w-a-\ve, z_1+y-a\right), z_3\in(z_1+b-a, \infty) \right\}.
$$

Moreover, these bounds naturally extend to the case where $\ve$ is any positive rational number. Indeed, let $\ve=\frac{m}{n}$, $m,n\in\mathbb N$. Then a similar reasoning as above leads to
\begin{align*}
& \eta\left(\left\{(z_1,z_2,z_3): z_1\in[a,a+m), z_2\in\left[z_1+w-a, z_1+y-a-\ve\right), z_3\in(z_1+b-a+\ve, \infty) \right\}\right)\\
\leq & \frac{m}{\ve}\eta([a,a+\ve)\times [w,y)\times (b+\ve, \infty))\\
\leq & \eta\left(\left\{(z_1,z_2,z_3): z_1\in[a,a+m), z_2\in\left[z_1+w-a-\ve, z_1+y-a\right), z_3\in(z_1+b-a, \infty) \right\}\right).
\end{align*}
Then by Corollary \ref{cor:invariance},
\begin{align*}
& \eta\left(\left\{(z_1,z_2,z_3): z_1\in[a,a+1), z_2\in\left[z_1+w-a, z_1+y-a-\ve\right), z_3\in(z_1+b-a+\ve, \infty) \right\}\right)\\
\leq & \frac{1}{\ve}\eta([a,a+\ve)\times [w,y)\times (b+\ve, \infty))\\
\leq & \eta\left(\left\{(z_1,z_2,z_3): z_1\in[a,a+1), z_2\in\left[z_1+w-a-\ve, z_1+y-a\right), z_3\in(z_1+b-a, \infty) \right\}\right)
\end{align*}
for any positive rational $\ve>0$. Since $\frac{1}{\ve}\eta([a,a+\ve)\times [w,y)\times (b+\ve, \infty))$ is continuous in $\ve$, by the continuity of measure, we have
\begin{multline*}
\frac{1}{\ve}\eta([a,a+\ve)\times [w,y)\times (b+\ve, \infty))\\
\to \eta \left(\{(z_1,z_2,z_3): z_1\in[a,a+1), z_2\in[z_1+w-a, z_1+y-a), z_3\in (z_1+b-a, \infty)\}\right)
\end{multline*}
as $\ve\to 0$. This is exactly $\mu^{(a,b)}|_{[u,v]}([w,y))$ defined in (\ref{e:mu}). The convergence to $\nu^{(a,b)}|_{[u,v]}$ can be shown symmetrically.
\end{proof}

We now prove the main result of this paper. Denote by $\mathring I$ the interior of the compact interval $I$, and recall that $\dot{\vf}^t(x)=\frac{\partial \vf(x,t)}{\partial t}|_{x,t}$. In addition, for any flow $\vf$ on $\mathbb R$ satisfying Assumptions (\ref{cond.C1}) and (\ref{cond.isolated}) and a given interval $[a,b]$ between two consecutive extended fixed points of $\vf$, introduce measures $\mu^{(a,b)}_\vf$ and $\nu^{(a,b)}_\vf$ as the pull-backs of $\mu^{(a,b)}$ and $\nu^{(a,b)}$ under the bijection $\tau$, which is defined using any given reference point $x_0$ between these two extended fixed points. More precisely, assuming that $\tau$ is increasing, then define measure $\mu^{(a,b)}_\vf$ on $(a,b)$ by
\begin{multline*}
\mu^{(a,b)}_\vf ([w,y)):=\eta \left((z_1,z_2,z_3): \tau(z_1)\in[\tau(a),\tau(a+1)),\right.\\
\left.\tau(z_2)\in[\tau(z_1)+\tau(w)-\tau(a), \tau(z_1)+\tau(y)-\tau(a)), \tau(z_3)\in (\tau(z_1)+\tau(b)-\tau(a), \infty)\right)
\end{multline*}
for all $w,y\in(a,b), w<y$. $\nu^{(a,b)}_\vf$ is defined similarly. The case where $\tau$ is decreasing is symmetric.

\begin{theorem}\label{thm:representation}
Let $\vf$ be a flow on $\mathbb R$ satisfying Assumptions (\ref{cond.C1}) and (\ref{cond.isolated}), and $L$ be a $\vf$-stationary intrinsic random location. Let $\alpha, \beta$ be two consecutive points in $\bar{\Phi}_0$. Then for any $I=[a,b]\subset (\alpha, \beta)$, the distribution of $L(I)$ is absolutely continuous in $\mathring I$, and it has a c\`{a}dl\`{a}g density function, denoted by $f$. Moreover, $f$ satisfies
\begin{equation}\label{e:signed}
\dot{\vf}^0(x_2) f(x_2)-\dot{\vf}^0(x_1) f(x_1)=\nu^{(a,b)}_\vf((x_1, x_2])-\mu^{(a,b)}_\vf((x_1, x_2])
\end{equation}
for any $x_1\leq x_2$, $x_1, x_2\in \mathring{I}$.
\end{theorem}

\begin{proof}
By Remark (\ref{rem:transform}), it suffices to prove the result for $\vf^t(x)=x+t$, where $\dot{\vf}^0(x)$ becomes the constant 1, and $\mu^{(a,b)}_\vf$ and $\nu^{(a,b)}_\vf$ are simply $\mu^{(a,b)}$ and $\nu^{(a,b)}$ defined before Proposition \ref{prop:limit}.

Let $C^\infty_C((u,v))$ be the set of smooth functions from $\bar{\mathbb R}$ to $\mathbb R$ with support in $(u,v)$, and $g$ be any function in $C^\infty_C((u,v))$. By stationarity, for any $\ve>0$, we have
$$
E[g(L([a+\ve, b+\ve]))]=E[g(L([a,b])+\ve)],
$$
hence
\begin{equation}\label{e:main}
E[g(L([a+\ve, b+\ve]))]-E[g(L([a, b]))]=E[g(L([a,b])+\ve)]-E[g(L([a, b]))].
\end{equation}

Denote by $F$ the distribution of $L([a,b])$, then the right hand side of (\ref{e:main}) can be rewritten as
$$
\int_a^b (g(s+\ve)-g(s))dF(s).
$$
Since $g$ is smooth and compactly supported, $g'$ is bounded, hence $g$ is uniformly Lipschitz. As a result, Dominated Convergence Theorem applies and we have
\begin{align}
& \lim_{\ve\to 0}\frac{1}{\ve}\left(E[g(L([a,b])+\ve)]-E[g(L([a, b]))]\right)\nonumber\\
= & \lim_{\ve\to 0}\frac{1}{\ve} \int_a^b (g(s+\ve)-g(s))dF(s)\nonumber\\
= & \int_a^b g'(s)dF(s)=\int_u^v g'(s)dF(s).\label{e:right}
\end{align}

For the left hand side of (\ref{e:main}), we have
\begin{align*}
& E[g(L([a+\ve, b+\ve]))]-E[g(L([a, b]))]\\
= & E[g(L([a+\ve, b+\ve])); L([a,b])\in [a,a+\ve)]-E[g(L([a, b])); L([a+\ve,b+\ve])\in (b,b+\ve]]\\
& +E[g(L([a+\ve, b+\ve])); L([a,b])\in [a+\ve, b]]-E[g(L([a, b])); L([a+\ve,b+\ve])\in [a+\ve,b]],
\end{align*}
where the notation $E[X;A]$ stands for the expectation of $X$ restricted on $A$, \textit{i.e.}, $E[X;A]=E[X\mathbf 1_A]$. Since $g$ is supported on $[u,v]\subset (a,b)$, for $\ve<u-a$,
\begin{align*}
& E[g(L([a+\ve, b+\ve])); L([a,b])\in [a+\ve, b]]\\
= & E[g(L([a+\ve, b+\ve])); L([a,b])\in [a+\ve, b], L([a+\ve, b+\ve])\in [a+\ve, b]]\\
= & E[g(L([a, b])); L([a,b])\in [a+\ve, b], L([a+\ve, b+\ve])\in [a+\ve, b]]\\
= & E[g(L([a, b])); L([a+\ve,b+\ve])\in [a+\ve,b]],
\end{align*}
where the equality in the middle comes from the stability under restriction property of $L$. Therefore, we have
\begin{align}
& E[g(L([a+\ve, b+\ve]))]-E[g(L([a, b]))]\nonumber\\
= &  E[g(L([a+\ve, b+\ve])); L([a,b])\in [a,a+\ve)]-E[g(L([a, b])); L([a+\ve,b+\ve])\in (b,b+\ve]]\nonumber\\
= & E[g(L([a+\ve, b+\ve])); L([a,b])\in [a,a+\ve), L([a+\ve, b+\ve])\in [u,v]]\nonumber\\
& -E[g(L([a, b])); L([a+\ve,b+\ve])\in (b,b+\ve], L([a,b])\in [u,v]]\label{e:expectation}\\
= & \int_u^v g(s)M_{\ve, [u,v]}d\mu_{\ve, [u,v]}(s)-\int_u^v g(s)N_{\ve, [u,v]}d\nu_{\ve, [u,v]}(s).\label{e:integral}
\end{align}

Combining (\ref{e:integral}) with Proposition \ref{prop:limit}, we have
\begin{align*}
&\lim_{\ve\to 0}\frac{1}{\ve}\left(E[g(L([a+\ve, b+\ve]))]-E[g(L([a, b]))]\right)\\
= & \int_u^v g(s)d(\mu^{(a,b)}-\nu^{(a,b)})(s),
\end{align*}
hence by (\ref{e:main}) and (\ref{e:right}),
$$
\int_u^v g'(s)dF(s)=\int_u^v g(s)d(\mu^{(a,b)}-\nu^{(a,b)})(s)
$$
for all $g\in C^\infty_C((u,v))$. This means, the signed measure on $(u,v)$ given by $d(\nu^{(a,b)}-\mu^{(a,b)})(s)$ is a derivative of the measure given by $dF(s)$ in the sense of generalized function. (Generalized functions are alternatively called distributions. In this paper we would use the term ``generalized functions'' to avoid confusion with the probability distributions of the random locations. Readers are referred to \cite{barros1973introduction} for an overview of the generalized functions.) Consequently, we have
$$
F((u,x])=\int_u^x \nu^{(a,b)}((u,s])-\mu^{(a,b)}((u,s])+c ~ds
$$
for all $x\in(u,v)$ and some constant $c$. Note that $c$ is inside the integral as it is a constant in the sense of generalized function. As a result, $F$ is differentiable on $(u,v)$; its derivative, denoted as $f$, satisfies
\begin{equation}\label{e:density}
f(x)=\nu^{(a,b)}((u,x])-\mu^{(a,b)}((u,x])+c,
\end{equation}
for almost all $x$ in $(u,v)$. It is easy to see that if we indeed define $f$ according to (\ref{e:density}) at every point $x\in(u,v)$, then such defined $f$ is still a version of the density, and $f$ is c\`{a}dl\`{a}g on $(u,v)$. Taking $u\downarrow a$ and $v\uparrow b$ shows that $F$ is absolutely continuous on $(a,b)$, and $f(x)=\nu^{(a,b)}((x_0,x])-\mu^{(a,b)}((x_0,x])+c$, $x\in(a,b)$ is a c\`{a}dl\`{a}g version of the density of $F$ on $(a,b)$. Here $x_0$ is an arbitrary fixed point in $(a,b)$, and $\nu^{(a,b)}((x_0,x])$ (resp. $\mu^{(a,b)}((x_0,x])$) is understood as $-\nu^{(a,b)}((x, x_0])$ (resp. $-\mu^{(a,b)}((x, x_0])$) when $x<x_0$. Moreover, taking $x=x_0$ leads to $c=f(x_0)$. Therefore, we have
$$
f(x)=f(x_0)+\nu^{(a,b)}((x_0,x])-\mu^{(a,b)}((x_0,x]), \quad x\in(a,b),
$$
or alternatively,
$$
f(x_2)-f(x_1)=\nu^{(a,b)}((x_1,x_2])-\mu^{(a,b)}((x_1, x_2]), \quad x_1, x_2\in(a,b), x_1\leq x_2.
$$

The proof is completed by applying the change of variable given in Remark (\ref{rem:transform}) for general flow $\vf$ satisfying Assumptions (\ref{cond.C1}) and (\ref{cond.isolated}).
\end{proof}

\begin{remark}
The significance of Theorem \ref{thm:representation} resides in the fact that while the left hand side of (\ref{e:signed}) is about the density function of a random location, the right hand side is the difference of two monotone functions. In other words, the probability density function of the random location behaves, surprisingly, more like the cumulative distribution function of a signed measure than a typical density function. Such a result is rare given that a density function is above all only uniquely defined in the almost sure sense. Hence, Theorem \ref{thm:representation} first guarantees the existence of a particular c\`{a}dl\`{a}g version of the density function, then equates it with the cumulative distribution function of a signed measure. Some consequences of the density function behaving as a cumulative distribution function can be seen in the following Corollaries.
\end{remark}

A simple rewrite of the result in Theorem \ref{thm:representation} gives rise to a conservation law when the interval of interest moves according to the flow $\vf$, which indicates clearly that what we obtained is, by nature, a Noether theorem. More precisely, consider a given interval $[a_0, b_0]$ between two consecutive extended fixed points, $\alpha$ and $\beta$, of $\vf$. Let $L$ be a $\vf$-stationary intrinsic random location. For any $x\in(\alpha,\beta)$ and $t\in\mathbb R$ such that $x\in(\vf^t(a_0), \vf^t(b_0))$, denote by $f_t(x)$ the density of $L([\vf^t(a_0), \vf^t(b_0)])$ at point $x$. Moreover, fix a reference point $x_0\in(a_0, b_0)$, and define the single-variable function $K(y)=\nu^{(a_0,b_0)}_\vf((x_0,y])-\mu^{(a_0,b_0)}_\vf((x_0,y])$ for $y\in(a_0, b_0)$, where $\nu^{(a_0,b_0)}_\vf((x_0,y])$ (resp. $\mu^{(a_0,b_0)}_\vf((x_0,y])$) is understood as $-\nu^{(a_0,b_0)}_\vf((y,x_0])$ (resp. $-\mu^{(a_0,b_0)}_\vf((y,x_0])$) for $y<x_0$. Then we have

\begin{corollary}
$$
\dot{\vf}^0(x)f_t(x)-K((\vf^t)^{-1}(x))
$$
is a constant in $t$ for $t$ satisfying $x\in(\vf^t(a_0), \vf^t(b_0))$.
\end{corollary}

\begin{proof}
Since $L$ is $\vf$-stationary, by the change of variable formula and (\ref{e:deriv}),
$$
\dot{\vf}^0(x)f_t(x)=f'_t(\tau(x))=f'_0(\tau(x)-t)=f'_0(\tau((\vf^t)^{-1}(x)))=\dot{\vf}^0((\vf^t)^{-1}(x))f_0((\vf^t)^{-1}(x)),
$$
where $f'_t$ is the density function of the stationary intrinsic random location $L'$ defined by
$$
L'(I)=\tau(L(\tau^{-1}(I)))
$$
on interval $I=[\tau(a_0)+t, \tau(b_0)+t]$.

By Theorem \ref{thm:representation}, we have
\begin{align*}
& \dot{\vf}^0((\vf^t)^{-1}(x))f_0((\vf^t)^{-1}(x))\\
=& \dot{\vf}^0(x_0)f_0(x_0)+\nu^{(a_0,b_0)}_\vf((x_0, (\vf^t)^{-1}(x)])-\mu^{(a_0,b_0)}_\vf((x_0, (\vf^t)^{-1}(x)])\\
=& \dot{\vf}^0(x_0)f_0(x_0)+K((\vf^t)^{-1}(x)),
\end{align*}
hence
$$
\dot{\vf}^0(x)f_t(x)-K((\vf^t)^{-1}(x))=\dot{\vf}^0(x_0)f_0(x_0),
$$
which is a constant in $t$ for $t$ satisfying $x\in(\vf^t(a_0), \vf^t(b_0))$.
\end{proof}

Also as a consequence of Theorem \ref{thm:representation}, we have the following result, which shows that the total variation of $\dot{\vf}^0(x)f(x)$ is bounded by its values and limits.

Denote by $TV^+_{(u,v)}(f), TV^-_{(u,v)}(f)$ and $TV_{(u,v)}(f)$ the positive variation, negative variation and total variation of the function $f$ on the interval $(u,v)$, respectively. That is,
$$
TV^+_{(u,v)}(f):=\sup_{u< x_1<\cdots<x_n< v}\sum_{i=1}^{n-1}(f(x_{i+1})-f(x_i))^+,
$$
$$
TV^-_{(u,v)}(f):=\sup_{u< x_1<\cdots<x_n< v}\sum_{i=1}^{n-1}(f(x_{i+1})-f(x_i))^-,
$$
and
$$
TV_{(u,v)}(f):=\sup_{u< x_1<\cdots<x_n< v}\sum_{i=1}^{n-1}|f(x_{i+1})-f(x_i)|,
$$
where the suprema are taken over all the partitions of $(u,v)$. Define $f(x-)=\lim_{y\uparrow x}f(y)$ to be the left limit of a c\`{a}dl\`{a}g function.

\begin{corollary}\label{cor:TV}
Let $\vf$ be a flow on $\mathbb R$ satisfying Assumptions (\ref{cond.C1}) and (\ref{cond.isolated}), and $L$ be a $\vf$-stationary intrinsic random location. Let $\alpha, \beta$ be two consecutive points in $\bar{\Phi}_0$. Then for any $I=[a,b]\in\mathcal I$ such that $I\subset (\alpha, \beta)$ and $u,v\in(a,b)$, $u<v$, the c\`{a}dl\`{a}g density function $f$ of $L(I)$ on $(a,b)$ satisfies
\begin{equation}\label{e:pos}
TV^+_{(u,v)}(\dot{\vf}^0(\cdot)f(\cdot))\leq \dot{\vf}^0(v)\min\{f(v), f(v-)\},
\end{equation}
\begin{equation}\label{e:neg}
TV^-_{(u,v)}(\dot{\vf}^0(\cdot)f(\cdot))\leq \dot{\vf}^0(u)\min\{f(u), f(u-)\},
\end{equation}
and
\begin{equation}\label{e:tot}
TV_{(u,v)}(\dot{\vf}^0(\cdot)f(\cdot))\leq \dot{\vf}^0(u)\min\{f(u), f(u-)\}+\dot{\vf}^0(v)\min\{f(v), f(v-)\}.
\end{equation}
\end{corollary}

\begin{remark}
One of the main results in \cite{samorodnitsky2013intrinsic} and \cite{shen2016random} was the so-called ``total variation constraint'', which states that the density $f$ of the distribution of a random location compatible with translation, for stationary or stationary increment processes, satisfies
$$
TV^+_{(u,v)}(f)\leq \min\{f(v), f(v-)\},
$$
$$
TV^-_{(u,v)}(f)\leq \min\{f(u), f(u-)\},
$$
and
$$
TV_{(u,v)}(f)\leq \min\{f(u), f(u-)\}+\min\{f(v), f(v-)\}.
$$
Now it becomes clear that they are special cases of Corollary \ref{cor:TV} where $\vf^t(x)=x+t$, hence consequences of the Noether theorem for random locations.
\end{remark}

The proof of Corollary \ref{cor:TV} mainly relies on the following proposition, which gives upper bounds for the mass that $\mu^{(a,b)}$ and $\nu^{(a,b)}$ can put on an interval. For simplicity, the proposition is presented using stationary intrinsic random locations. It is straightforward to extend all the definitions and results to general $\vf$-stationary intrinsic random locations if needed.

\begin{proposition}\label{prop:totalmass}
Let $L$ be a stationary intrinsic random location. Under the same setting as before, $\mu^{(a,b)}([u,v])\leq f(u-)$, $\nu^{(a,b)}([u,v])\leq f(v)$.
\end{proposition}

\begin{proof}
Take $v'\in(v,b)$, then
$$
\mu^{(a,b)}([u,v])\leq \mu^{(a,b)}([u,v'))=\lim_{\ve\to 0}\frac{1}{\ve}M_{\ve, [u,v']}\mu_{\ve, [u,v']}([u,v'))\leq\limsup_{\ve\to 0}\frac{1}{\ve}M_{\ve, [u,v']},
$$
since $\mu_{\ve, [u,v']}$ is a probability measure.

On the other hand, by definition, for $\epsilon<u-a$,
\begin{align*}
M_{\ve, [u,v']} & =P(L([a,b])\in[a,a+\ve), L[a+\ve, b+\ve]\in[u,v'])\\
& \leq P(L([a,b])\in[a,a+\ve), L[a+\ve, b]\in[u,v'])\\
& = P(L([a+\ve, b])\in [u,v'])-P(L([a,b])\in [u,v']).
\end{align*}
Moreover, by (\ref{e:boundbyf}) we have, for $\ve$ small enough,
$$
P(L([a+\ve, b])\in [u,v'])-P(L([a,b])\in [u,v'])\leq P(L([a,b])\in[u-\ve, u)).
$$
Hence
$$
\limsup_{\ve\to 0}\frac{1}{\ve}M_{\ve,[u,v']}\leq \lim_{\ve\to 0}\frac{1}{\ve}P(L([a,b])\in[u-\ve, u))=f(u-).
$$
The bound for $\nu^{(a,b)}([u,v])$ can be derived symmetrically.
\end{proof}

\begin{proof}[Proof of Corollary \ref{cor:TV}]
For simplicity we only prove the result for $\vf^t(x)=x+t$. The general case then follows by the change of variable discussed in Remark \ref{rem:transform}.

In this case, by Theorem \ref{thm:representation}, we have
$$
f(x_2)-f(x_1)= \nu^{(a,b)}((x_1, x_2])-\mu^{(a,b)}((x_1, x_2])
$$
for any $x_1, x_2\in [u,v], x_1<x_2$.

Hence
$$
(f(x_2)-f(x_1))^+\leq \nu^{(a,b)}((x_1, x_2]).
$$

Therefore, for any partition $u< x_1<\cdots <x_n< v$ of $(u,v)$,
$$
\sum_{i=1}^{n-1}(f(x_{i+1})-f(x_i))^+\leq\nu^{(a,b)}((u, v])\leq f(v)
$$
by Proposition \ref{prop:totalmass}. Taking supremum over all partitions of $(u,v)$ on the left hand side leads to
$$
TV^+_{(u,v)}(f)\leq f(v).
$$
Moreover, since $f$ is c\`{a}dl\`{a}g, we also have
$$
TV^+_{(u,v)}(f)=\lim_{y\uparrow v}TV^+_{(u,y)}(f)\leq \lim_{y\uparrow v}f(y)=f(v-),
$$
hence
$$
TV^+_{(u,v)}(f)\leq \min\{f(v-), f(v)\}.
$$
The result for $TV^-_{(u,v)}(f)$ can be proved symmetrically. Finally, adding the two inequalities (\ref{e:pos}) and (\ref{e:neg}) gives (\ref{e:tot}).
\end{proof}

\section{Boundary and near-boundary behavior}

In Section 3, we mainly focus on the behavior of the distribution of a $\vf$-stationary intrinsic random location $L$ in the interior of the interval of interest $I=[a,b]$. We have seen that a c\`{a}dl\`{a}g density, denoted by $f$, exists on $(a,b)$. Indeed, (\ref{abso}) gives an upper bound for $f(x), x\in(a,b)$. Such a bound, however, diverges as $x$ approaches $a$ or $b$. Moreover, there may also be point masses on the two boundaries of the interval, which were not studied in Section 3. In this section we provide these missing pieces by discussing the boundary and near-boundary behavior of $L$.

For simplicity, in this section we always assume that $L$ is a stationary intrinsic random location. The results can be easily generalized to the case where $L$ is $\vf$-stationary.

Recall that $S=\{x\in\mathbb R: x=L(I) \text{ for some } I\in\mathcal I\}$, $l_x=\sup\{y\in S: y<x, x\preceq y\}$ and $r_x=\inf\{y\in S: y>x, x\preceq y\}$, where ``$\preceq$'' is the partial order determined by $L$. For any $T>0$, define $S_{l,T}:=\{x\in S: l_x=x, r_x\geq x+T\}$ and $S_{r,T}:=\{x\in S: r_x=x, l_x\leq x-T\}$. Denote by $Leb(\cdot)$ the Lebesgue measure on $\mathbb R$. Then we have
\begin{proposition}\label{prop:boundary}
For $I=[a,b]$,
\begin{equation}\label{e:boundaryl}
P(L(I)=a)=P(a\in S_{l,b-a})=E(Leb(S_{l,b-a}\cap [0,1))),
\end{equation}
\begin{equation}\label{e:boundaryr}
P(L(I)=b)=P(b\in S_{r,b-a})=E(Leb(S_{r,b-a}\cap[0,1))).
\end{equation}
\end{proposition}

\begin{proof}
By symmetry, it suffices to prove \ref{e:boundaryl}. Note that for $x\in S$, $l_x=x, r_x>x+b-a$ implies that $L([x,x+b-a])=x$, which in turn implies that $l_x\leq x, r_x\geq x+b-a$. Hence we have
$$
P(a\in S, l_a=a, r_a>b)\leq P(L([a,b])=a)\leq P(a\in S, l_a\leq a, r_a\geq b).
$$
However,
$$
P(a\in S, l_a<a, r_a\geq b)=\eta((-\infty, a)\times\{a\}\times [b,\infty))=0,
$$
since the plane with the second coordinate fixed is a $\eta-$null set, according to Theorem \ref{thm:relation}. Therefore for $\epsilon>0$,
\begin{multline}\label{e:boundaryine}
P(L([a,b+\epsilon])=a)\leq P(a\in S, l_a=a, r_a>b)\\
\leq P(L([a,b])=a)\leq P(a\in S, l_a=a, r_a\geq b)\leq P(L([a,b-\epsilon])=a).
\end{multline}
Next, by a similar reasoning as in the proof of Lemma \ref{lem:continuity}, $P(L([a,b])=a)$ is continuous in $b$ for $b>a$. Indeed, for $b'>b$,
\begin{align*}
& P(L([a,b])=a)-P(L([a,b'])=a)\\
= & P(L([a,b'])\in[b,b'))-[P(L([a,b]\in(a,b)))-P(L([a,b'])\in(a,b))]\\
 & -[P(L([a,b])=b)-P(L([a,b'])=b')]-[P(L([a,b])=\infty)-P(L([a,b'])=\infty)]\\
\leq & P(L([a,b'])\in[b,b'))\leq P(L([a+b'-b,b'])\in[b,b'))=P(L([a,b])\in[2b-b',b))\to 0
\end{align*}
as $b'\downarrow b$, where the inequalities follow from Lemma \ref{lem:comparison}, and the convergence is due to the existence of a density of $L([a,b])$ on $(a,b)$ given by Theorem \ref{thm:representation}.

 Thus, we have $P(L(I)=a)=P(a\in S_{l,b-a})$ by taking $\epsilon\to 0$ in (\ref{e:boundaryine}) and applying the continuity result proved above. The second equality in (\ref{e:boundaryl}) then follows naturally by the observation that $P(x\in S_{l,b-a})$ is constant in $x$, due to the equality $P(L(I)=a)=P(a\in S_{l,b-a})$ and the fact that $L$ is a stationary intrinsic random location.
\end{proof}

We now turn to the near-boundary behavior of the distribution of $L(I)$, $I=[a,b]$. More precisely, we would like to know when the density $f(x)$ will explode as $x$ approaches the boundaries of the interval $I$. Clearly, by Theorem \ref{thm:representation} and Proposition \ref{prop:totalmass}, $\lim_{x\downarrow a}f(x)=\infty$ if and only if $\mu^{(a,b)}((a,x_0])=\infty$ for some (equivalently, any) $x_0\in (a,b)$. By (\ref{e:mu}), this means
\begin{equation}\label{e:leftexplode}
\eta(\{(z_1, z_2, z_3): z_1\in[a,a+1), z_2\in (z_1, z_1+x_0-a], z_3\in (z_1+b-a,\infty)\})=\infty.
\end{equation}
Similarly, $\lim_{x\uparrow b}f(x)=\infty$ if any only if
\begin{equation}\label{e:rightexplode}
\eta(\{(z_1, z_2, z_3): z_1\in(-\infty,z_3+a-b), z_2\in (z_3+x_0-b, z_3), z_3\in (b,b+1]\})=\infty.
\end{equation}

Define set
$$
S_1:=\{x\in[0,1): l_x<x, r_x>x, r_x-l_x>b-a\},
$$
then (\ref{e:leftexplode}) or (\ref{e:rightexplode}) would require $E(|S_1|)=\infty$, where $|S_1|$ is the cardinal number of $S_1$, with the convention that $|\cdot|=\infty$ for any infinite set. Indeed, by Corollary \ref{cor:invariance} and taking $x_0=\min\{a+\frac{1}{2}, b\}$, (\ref{e:leftexplode}) holds if and only if
\begin{equation}\label{e:infinity}
2\eta(\{(z_1, z_2, z_3): z_1\in\left[a,a+1/2\right), z_2\in \left(z_1, z_1+\min\{1/2, b-a\}\right], z_3\in (z_1+b-a,\infty)\})=\infty.
\end{equation}
Since the set $\{(z_1, z_2, z_3): z_1\in\left[a,a+1/2\right), z_2\in \left(z_1, z_1+\min\{1/2, b-a\}\right], z_3\in (z_1+b-a,\infty)\}$ is a subset of $\{(z_1, z_2, z_3): z_1<z_2, z_2\in[a,a+1), z_3>z_2, z_3-z_1>b-a\}$, (\ref{e:infinity}) implies that the latter set must also have measure $\infty$ under $\eta$. Then
$$
\eta(\{(z_1, z_2, z_3): z_1<z_2, z_2\in[0,1), z_3>z_2, z_3-z_1>b-a\})=E(|S_1|)=\infty
$$
by Corollary \ref{cor:invariance}.

Although not a necessary condition, one direct and simple way leading to $E(|S_1|)=\infty$ is, of course, to have $S_1$ to be an infinite set with positive probability. The next proposition gives a necessary and sufficient condition for $S_1$ to be infinite.

\begin{proposition}\label{prop:explode}
The set $S_1$ has infinite number of elements if and only if at least one of the following four scenarios is true:
\begin{enumerate}[(1)]
\item There exists an increasing sequence $\{x_n\}_{n=1,2,...}$ in $S\cap [0,1)$, such that for each $n$, $x_{n+1}\preceq x_n$, $l_{x_n}<x_n$, and $r_{x_n}\geq x_n+b-a$;
\item There exists an decreasing sequence $\{x_n\}_{n=1,2,...}$ in $S\cap [0,1)$, such that for each $n$, $x_{n}\preceq x_{n+1}$, $l_{x_n}<x_n$, and $r_{x_n}\geq x_n+b-a$;
\item There exists an decreasing sequence $\{x_n\}_{n=1,2,...}$ in $S\cap [0,1)$, such that for each $n$, $x_{n+1}\preceq x_n$, $r_{x_n}>x_n$, and $l_{x_n}\leq x_n-b+a$;
\item There exists an increasing sequence $\{x_n\}_{n=1,2,...}$ in $S\cap [0,1)$, such that for each $n$, $x_{n}\preceq x_{n+1}$, $r_{x_n}>x_n$, and $l_{x_n}\leq x_n-b+a$.
\end{enumerate}
\end{proposition}

\begin{proof}
The ``if'' part is trivial. For the ``only if'' part, assume $|S_1|=\infty$. Then there exists a monotone sequence of points in $S_1$. Without loss of generality, assume the sequence is increasing, and denote it by $\{x_n\}_{n=1,2,...}$, with $\lim_{n\to\infty}x_n=x_{\infty}$, which is not necessarily in $S_1$. Moreover, $x_1$ can be chosen so that $x_\infty-x_1<b-a$.

Next, the sequence can be taken such that for any $n=1,2,...$, either $x_n\preceq x_{n+1}$ or $x_n\succeq x_{n+1}$, which is not trivial since ``$\preceq$'' is only a partial order. To see this, consider the set of indices $J=\{j: x_j \npreceq x_{j+1}, x_{j+1}\nsucceq x_{j}\}$. For any $n\in J$, let $y_n=L([x_n, x_{n+1}])$, then $y_n\in (x_n,x_{n+1})$. As such, we have $r_{x_n}\leq y_n$. By the definition of $S_1$, this implies that $l_{y_n}\leq l_{x_n}<r_{x_n}-(b-a)\leq y_n-(b-a)$. Symmetrically, $r_{y_n}>y_n+(b-a)$. This means, for any $n_1, n_2\in J$, $|y_{n_1}-y_{n_2}|\geq b-a$, which guarantees that $J$ is a finite set. Taking the subsequence of $\{x_n\}$ starting from $n_0=\max\{j: j\in J\}+1$ gives a new sequence for which either $x_n\preceq x_{n+1}$ or $x_n\succeq x_{n+1}$.

For such a sequence, it is clear that for any $n\geq 2$, $x_n\preceq x_{n-1}$ and $x_n\preceq x_{n+1}$ can not hold at the same time, since otherwise $r_{x_n}-l_{x_n}\leq x_{n+1}-x_{n-1}<b-a$, implying that $x_n$ can not be in $S_1$. Thus, either $\{x_n\}_{n=1,2,...}$ is monotone according to $\preceq$, or there exists $n_0$, such that $x_1\preceq x_2\preceq \cdots \preceq x_{n_0}$ and $x_{n_0}\succeq x_{n_0+1}\succeq \cdots$. As a result, there always exists a subsequence of $\{x_n\}_{n=1,2,...}$, still denoted as $\{x_n\}_{n=1,2,...}$ by a slight abuse of notation, which is monotone according to $\preceq$. Next we discuss the two possible cases.

Case 1: $x_{n+1}\preceq x_n$ for any $n$. In this case note that $l_{x_n}\in[x_{n-1}, x_n)$, hence $\lim_{n\to\infty}l_{x_n}=x_\infty$. Moreover, since $x_n$ is decreasing in $n$ according to $\preceq$ and $r_{x_n}>l_{x_n}+b-a\geq x_{n-1}+b-a >x_\infty$ for any $n\geq 2$, $r_{x_n}$ is non-increasing in $n$ for $n\geq 2$. Therefore,
$$
r_{x_n}\geq \lim_{n\to\infty}r_{x_n}\geq \lim_{n\to\infty} l_{x_n}+b-a = x_\infty+b-a>x_n+b-a
$$
for $n=2,...$. Thus, scenario (1) in the proposition holds for $\{x_n\}_{n=2,3,...}$.

Case 2: $x_{n}\preceq x_{n+1}$ for any $n$. Then $r_{x_1}<x_\infty$, hence $l_{x_1}<x_\infty-b+a$. By a similar reasoning as in case 1, $l_{x_n}$ is non-increasing in $n$, so $l_{x_n}\leq l_{x_1}$. Recall that $x_n$ is increasing and $\lim_{n\to\infty}x_n=x_{\infty}$, therefore, there exists $n_0$, such that $x_\infty-x_{n}<x_\infty-b+a-l_{x_1}$ for any $n>n_0$, which implies $l_{x_n}\leq l_{x_1}<x_n-b+a$ for $n\geq n_0$. Taking the subsequence of $\{x_n\}_{n=1,2,...}$ starting from $x_{n_0}$ leads to scenario (4).

Scenarios (2) and (3) can be derived symmetrically by assuming that the sequence $\{x_n\}_{n=1,2,...}$ is decreasing.
\end{proof}

With Proposition \ref{prop:explode} proved, it is obvious that scenarios (1) and (2) corresponds to the explosion of the density $f$ near the boundary $a$, while scenarios (3) and (4) corresponds to the explosion of $f$ near the boundary $b$.

\begin{corollary}\label{cor:densityexplode}
Under the same setting as in Proposition \ref{prop:explode}, if (1) or (2) happens with positive probability, then $\lim_{x\downarrow a}f(x)=\infty$; if (3) or (4) happens with positive probability, then $\lim_{x\uparrow b} f(x)=\infty$.
\end{corollary}

\begin{proof}
We prove that scenario (1) implies $\lim_{x\downarrow a}f(x)=\infty$. The other cases are similar.

In scenario (1), for any $n\geq 2$, $0\leq x_{n-1}\leq l_{x_n}<x_n<1$, and $r_{x_n}>x_n+b-a>l_{x_n}+b-a$. Moreover, $x_n-l_{x_n}\leq x_n-x_{n-1}\to 0$ as $n\to\infty$. Hence scenario (1) happens with positive probability implies that
$$
\eta(\{(z_1, z_2, z_3): z_1\in[z_2-\Delta,z_2), z_2\in [0,1), z_3\in (z_1+b-a,\infty)\})=\infty
$$
for any $\Delta>0$. In particular,
$$
\eta(\{(z_1, z_2, z_3): z_1\in[z_2-x_0+a,z_2), z_2\in [0,1), z_3\in (z_1+b-a,\infty)\})=\infty \text{ for any } x_0\in(a,b).
$$
Note that
\begin{align*}
& \{(z_1, z_2, z_3): z_1\in[z_2-x_0+a,z_2), z_2\in [0,1), z_3\in (z_1+b-a,\infty)\}\\
\subset & \{(z_1, z_2, z_3): z_1\in[-x_0+a,1), z_2\in (z_1, z_1+x_0-a], z_3\in (z_1+b-a,\infty)\}.
\end{align*}
Thus (\ref{e:leftexplode}) holds:
\begin{align*}
& \eta(\{(z_1, z_2, z_3): z_1\in[a,a+1), z_2\in (z_1, z_1+x_0-a], z_3\in (z_1+b-a,\infty)\})\\
= & \frac{1}{1+x_0-a}\eta(\{(z_1, z_2, z_3): z_1\in[-x_0+a,1), z_2\in (z_1, z_1+x_0-a], z_3\in (z_1+b-a,\infty)\})\\
= & \infty,
\end{align*}
where the first equality follows from Corollary \ref{cor:invariance}.
\end{proof}

As an application of Proposition \ref{prop:boundary} and Corollary \ref{cor:densityexplode}, consider the location of the path supremum of a stochastic process $\mathbf X=\{X(t)\}_{t\in\mathbb R}$ with continuous sample paths, formally defined as
$$
\tau_{\mathbf X, I}:=\inf\{t\in I: X(t)=\sup_{s\in I}X(s)\}.
$$
The infimum is used to choose the leftmost point among all the points where $\sup_{s\in I}X(s)$ is achieved, in the case where there are more than one such point. If we further assume that

\textbf{Assumption U}. For any $I\in\mathcal I$,
$$
P(\text{there exist } t_1, t_2\in I, t_1\neq t_2, \text{ such that }X(t_1)=X(t_2)=\sup_{s\in I}X(s))=0,
$$
\textit{i.e.}, the location of the path supremum is almost surely unique, then the infimum in the definition of $\tau_{\mathbf X, I}$ can be removed.

Most of the commonly used processes do satisfy Assumption U. It is proved in \cite{kim1990cube} that for a Gaussian process $\mathbf X$, Assumption U holds if and only if $Var(X(t)-X(s))\neq 0$ for any $s\neq t$. A necessary and sufficient condition for more general processes with continuous sample paths can be found in \cite{pimentel2014location}.

Note that in the case of the location of the path supremum, the random set $S$, as defined before Lemma \ref{lem:partialorder}, takes the form
$$
S=\{t: \text{ there exists }\Delta>0, \text{ such that }X(t)=\sup_{s\in[t-\Delta, t]}X(s) \text{ or } X(t)=\sup_{s\in[t, t+\Delta]}X(s)\},
$$
and the partial order $\preceq$ is the natural order for the value of the process $\{X(t)\}_{t\in \mathbb R}$.

\begin{corollary}\label{cor:max}
Let $\mathbf X=\{X(t)\}_{t\in\mathbb R}$ be a stochastic process with continuous sample paths and stationary increments. Assume $\mathbf X$ satisfies Assumption U. If the local maxima of $\mathbf X$ is dense in $[a,b]$ with positive probability, then the density of $\tau_{\mathbf X, I}$, denoted by $f$, satisfies $\lim_{t\downarrow a}f(t)=\infty$ or $\lim_{t\uparrow b}f(t)=\infty$.
\end{corollary}

\begin{proof}
By the stationarity of the increments, it suffices to prove the results for the case where $a=0$. Denote by $D$ the event that the local maxima of $\mathbf X$ is dense. Let $\tau'=\tau_{\mathbf X, [0,4b]}$, then $r_{\tau'}-l_{\tau'}\geq 4b$. Therefore, $P(D, r_{\tau'}-\tau'\geq 2b)>0$ or $P(D, \tau'-l_{\tau'}\geq 2b)>0$. Without loss of generality, assume that $P(D, r_{\tau'}-\tau'\geq 2b)>0$. As a result,
$$
P(D, \text{there exists }t\in S\cap [0,4b], r_t-t\geq 2b)>0,
$$
hence also
$$
P(D, \text{there exists }t\in S\cap [0,b), r_t-t\geq 2b)>0
$$
by the stationarity of the increments.

Let $t_\infty=\tau_{\mathbf X, [0,b]}$. From now on we focus on the event
$$
\{D, \text{ there exists }t\in S\cap [0,b), r_t-t\geq 2b\}.
$$
In this case, $t_\infty=\tau_{\mathbf X, [0,2b]}<b$, and $r_{t_\infty}\geq 2b$. By Assumption U, there exists $\epsilon\in(0,b-t_\infty)$, such that $\inf_{s\in [t_\infty, t_\infty+\epsilon]}X(s)>\sup_{s\in[b, 2b]}X(s)$. For $n=1,2,...$, let $t_n=\tau_{\mathbf X, [t_\infty+\frac{1}{n+1}\epsilon, 2b]}$. Then $t_n$ is a non-increasing sequence satisfying $\lim_{n\to\infty}t_n=t_\infty$, and $X(t_n)\leq X(t_{n+1})$ for all $n$. Moreover, since
$$
\sup_{s\in [t_\infty+\frac{1}{n+1}\epsilon, b]}X(s)\geq \sup_{s\in [t_\infty+\frac{1}{n+1}\epsilon, t_\infty+\epsilon]}X(s)>\sup_{s\in[b, 2b]}X(s),
$$
$t_n\in [t_\infty+\frac{1}{n+1}\epsilon, b]$, and $r_{t_n}\geq 2b\geq t_n+b$. By removing all equal terms in $\{t_n\}_{n=1,2,...}$ and all the terms in $\{t_n\}_{n=1,2,...}$ at which the values of $\mathbf X$ are equal, we get a decreasing sequence $\{t_n\}_{n=1,2,...}$, satisfying $\lim_{n\to\infty} t_n=t_\infty$ and $X(t_n)<X(t_{n+1})$, hence $t_n\preceq t_{n+1}$, for all $n$. Since the local maxima are dense and the sample paths are continuous, such a sequence can be approached by a sequence of local maxima $\{t'_n\}_{n=1,2,...}$, while all the properties derived above still hold. In addition, as all the points in the new sequence are local maxima, we have $l_{t'_n}<t'_n$, $n=1,2,...$. By the stationarity of the increments, this is scenario (2) in Proposition \ref{prop:explode}. Symmetrically, if $P(\tau'-l_{\tau'}\geq 2b)>0$, then scenario (4) in Proposition \ref{prop:explode} happens with positive probability.
\end{proof}

The following result is a direct application of Corollary \ref{cor:max} and Proposition \ref{prop:boundary}. It applies to Brownian motions, Ornstein-Uhlenbeck processes, or more generally, any process $\{X(t)\}_{t\geq 0}$ satisfying Assumption U and of the form
$$
X(t)=\int_0^t Y(s)dB_s,
$$
where $\{Y(t)\}_{t\geq 0}$ is a predictable stationary process which is independent of the standard Brownian motion $\{B_t\}_{t\geq 0}$, and for which the above stochastic integral is well-defined.

\begin{corollary}
Let $\mathbf X=\{X(t)\}_{t\geq 0}$ be a continuous semimartingale with stationary increments, satisfying Assumption U. Assume that the local martingale part of $\mathbf X$ almost surely is not constant on any interval. For any $I=[a,b]\in\mathcal I$, let $\tau_{\mathbf X, I}$ be defined as previously, and $f$ be its density on $(a,b)$. Then $P(\tau_{\mathbf X, I}=a)=P(\tau_{\mathbf X, I}=b)=0$, and $\lim_{t\downarrow a}f(t)+\lim_{t\uparrow b}f(t)=\infty$.
\end{corollary}

\begin{proof}
Since $\mathbf X$ is a semimartingale and has a local martingale part which is nowhere flat, it is of unbounded variation over any interval, hence the local maxima and the local minima of $\mathbf X$ are almost surely dense in any interval. Thus, Corollary \ref{cor:max} applies. Moreover, since $a$ is almost surely an accumulation point, both from the left and from the right, of the level set $\{t\in\mathbb R: X(t)=X(a)\}$, for any $\epsilon>0$ there exists $t\in(a,a+\epsilon]$ such that $X(t)\geq X(a)$. If the equality holds for all such $t\in(a,b]$, then Assumption U is violated. Hence almost surely there exists $t\in (a,b]$ such that $X(t)> X(a)$. Thus,
$P(\tau_{\mathbf X, I}=a)=0$. The case for the right boundary $b$ is symmetric.
\end{proof}

\section*{Acknowledgement} Jie Shen acknowledges financial support from the China Scholarship Council. Yi Shen acknowledges financial support from the Natural Sciences and Engineering Research Council of Canada (RGPIN-2014-04840).

\bibliographystyle{apalike}
\bibliography{Bibtexp}

\begin{thebibliography}{}

\bibitem[Baez and Fong, 2013]{Baez2013}
Baez, J.~C. and Fong, B. (2013).
\newblock A {N}oether theorem for {M}arkov processes.
\newblock {\em Journal of Mathematical Physics}, 54(1):013301.

\bibitem[Barros-Neto, 1973]{barros1973introduction}
Barros-Neto, J. (1973).
\newblock {\em An introduction to the theory of distributions}.
\newblock Dekker.

\bibitem[Embrechts and Maejima, 2002]{embrechts2002selfsimilar}
Embrechts, P. and Maejima, M. (2002).
\newblock {\em Selfsimilar processes}.
\newblock Princeton Univ. Press.

\bibitem[Gough et~al., 2015]{gough2015noether}
Gough, J.~E., Ratiu, T.~S., and Smolyanov, O.~G. (2015).
\newblock Noether's theorem for dissipative quantum dynamical semi-groups.
\newblock {\em Journal of Mathematical Physics}, 56(2):022108.

\bibitem[Kim and Pollard, 1990]{kim1990cube}
Kim, J. and Pollard, D. (1990).
\newblock Cube root asymptotics.
\newblock {\em The Annals of Statistics}, pages 191--219.

\bibitem[Kosmann-Schwarzbach, 2011]{kosmann2011noether}
Kosmann-Schwarzbach, Y. (2011).
\newblock {\em The {N}oether Theorems}.
\newblock Springer.

\bibitem[Lezcano and de~Oca, 2017]{lezcano2017stochastic}
Lezcano, A.~G. and de~Oca, A. C.~M. (2017).
\newblock A stochastic version of the {N}oether theorem.
\newblock {\em arXiv preprint arXiv:1709.09295}.

\bibitem[Misawa, 1994]{misawa1994conserved}
Misawa, T. (1994).
\newblock Conserved quantities and symmetry for stochastic dynamical systems.
\newblock {\em Physics Letters A}, 195(3-4):185--189.

\bibitem[Noether, 1918]{noether1918}
Noether, E. (1918).
\newblock Invariant variation problems.
\newblock {\em Gott. Nachr.}, 1918:235--257.
\newblock [Transp. Theory Statist. Phys.1,186(1971)].

\bibitem[Pimentel, 2014]{pimentel2014location}
Pimentel, L.~P. (2014).
\newblock On the location of the maximum of a continuous stochastic process.
\newblock {\em Journal of Applied Probability}, 51(1):152--161.

\bibitem[Samorodnitsky and Shen, 2013]{samorodnitsky2013intrinsic}
Samorodnitsky, G. and Shen, Y. (2013).
\newblock Intrinsic location functionals of stationary processes.
\newblock {\em Stochastic Processes and their Applications},
  123(11):4040--4064.

\bibitem[Shen, 2013]{shen2013thesis}
Shen, Y. (2013).
\newblock {\em Stationarity and random locations}.
\newblock PhD thesis, Cornell University.

\bibitem[Shen, 2016]{shen2016random}
Shen, Y. (2016).
\newblock Random locations, ordered random sets and stationarity.
\newblock {\em Stochastic Processes and their Applications}, 126(3):906--929.

\bibitem[Shen, 2018]{shen2018self}
Shen, Y. (2018).
\newblock Location of the path supremum for self-similar processes with
  stationary increments.
\newblock {\em Annales de l'Institut Henri Poincar{\'e}, Probabilit{\'e}s et
  Statistiques}, 54(4):2349--2360.

\bibitem[Thieullen and Zambrini, 1997a]{thieullen1997probability}
Thieullen, M. and Zambrini, J.-C. (1997a).
\newblock Probability and quantum symmetries. i. the theorem of {N}oether in
  {S}chr{\"o}dinger's {E}uclidean quantum mechanics.
\newblock {\em Annales de l'Institut Henri Poincar\'{e} Physique
  Th{\'e}orique}, 67(3):297--338.

\bibitem[Thieullen and Zambrini, 1997b]{thieullen1997symmetries}
Thieullen, M. and Zambrini, J.-C. (1997b).
\newblock Symmetries in the stochastic calculus of variations.
\newblock {\em Probability theory and related fields}, 107(3):401--427.

\bibitem[Van~Casteren, 2003]{van2003hamilton}
Van~Casteren, J.~A. (2003).
\newblock The {H}amilton-{J}acobi-{B}ellman equation and the stochastic
  {N}oether theorem.
\newblock In {\em Evolution Equations: Applications to Physics, Industry, Life
  Sciences and Economics}, pages 375--401. Springer.

\bibitem[Yasue, 1981]{yasue1981stochastic}
Yasue, K. (1981).
\newblock Stochastic calculus of variations.
\newblock {\em Journal of functional Analysis}, 41(3):327--340.

\end{thebibliography}

\end{document}